\documentclass[12pt, a4paper]{amsart}

\usepackage[T1]{fontenc}
\usepackage[utf8]{inputenc}
\usepackage[top=2.5cm, bottom=2.5cm, left=2cm, right=2cm]{geometry}
\usepackage{graphicx}
\usepackage{float}
\usepackage[colorlinks=true, linkcolor=blue]{hyperref}
\usepackage{anyfontsize}
\usepackage{bbm}
\usepackage{amsmath}
\usepackage{amsthm}
\usepackage{amssymb}
\usepackage[shortlabels]{enumitem}
\usepackage[noabbrev]{cleveref}
\theoremstyle{plain}
\newtheorem{theorem}{Theorem}[section]
\newtheorem{fact}[theorem]{Fact}

\newtheorem{corollary}[theorem]{Corollary}

\newtheorem{question}{Question}

\theoremstyle{definition}

\newtheorem{definition}{Definition}

\newtheorem*{remark*}{Remark}
\newtheorem*{question*}{Question}

\DeclareMathOperator{\baire}{\omega^\omega}

\DeclareMathOperator{\seqN}{\omega^{<\omega}}
\DeclareMathOperator{\seq2}{2^{<\omega}}
\DeclareMathOperator{\analytic}{\mathbf{\Sigma}_1^1}
\DeclareMathOperator{\coanalytic}{\mathbf{\Pi}_1^1}
\DeclareMathOperator{\RLleq}{\leq_{RL}}
\DeclareMathOperator{\RLless}{<_{RL}}
\DeclareMathOperator{\RLmore}{>_{RL}}
\DeclareMathOperator{\lexq}{\leq_{lex}}
\DeclareMathOperator{\lex}{<_{lex}}
\DeclareMathOperator{\trees}{Tr_\omega}
\DeclareMathOperator{\illfounded}{IF_\omega}

\title{Sequences with increasing subsequence}

\author{Łukasz Mazurkiewicz}
\email{lukasz.mazurkiewicz@pwr.edu.pl}

\author{Szymon Żeberski}
\email{szymon.zeberski@pwr.edu.pl}

\thanks{The work has been partially financed by grant {\bf 8211204601, MPK: 9130730000} from the Faculty of Pure and Applied Mathematics, Wrocław University of Science and Technology.
	}

\address{Łukasz Mazurkiewicz, Szymon Żeberski, Faculty of Pure and Applied Mathematics, Wrocław University of Science and Technology, 50-370 Wrocław, Poland}

\begin{document}
	\maketitle
	\begin{abstract}
    We study analytic and Borel subsets defined similarily to the old example of analytic complete set given by Luzin. Luzin's example, which is essentially a subset of the Baire space, is based on the natural partial order on naturals, i.e. division. It consists of sequences which contain increasing subsequence in given order. 
    
    We consider a variety of sets defined in a similar way. Some of them occurs to be Borel subsets of the Baire space, while others are analytic complete, hence not Borel.

    In particular, we show that an analogon of Luzin example  based on the natural linear order on rationals is analytic complete.  We also  characterise all countable linear orders having such property.
\end{abstract}

\section{Introduction}

We will use standard set theoretic notions (mostly following \cite{srivastava} and \cite{jech}). In particular, $\omega$ is the first infinite cardinal, i.e. the set of all natural numbers, $\mathbb{N}=\omega\setminus\{0\}$ is the set of all positive natural numbers, $\seqN$ and $\seq2$ are sets of finite sequences of elements of $\omega$ and $\{0,1\}$, respectively.
  For this part assume that $\mathcal{X}$ and $\mathcal{Y}$ are Polish spaces, i.e. separable completely metrizable topological spaces. Classical examples are the real line $\mathbb{R}$, the Baire space $\omega^\omega$, $\mathbb{N}^\omega$, the Cantor space $2^\omega$.
\begin{definition}
    We say that $A\subseteq \mathcal{X}$ is \textbf{$\analytic$-complete} if $A$ is analytic and for every Polish space $\mathcal{Y}$ and every analytic $B\subseteq \mathcal{Y}$ there is a Borel map $f: \mathcal{Y}\rightarrow \mathcal{X}$ such that $f^{-1}(A)=B$.
\end{definition}
\begin{definition}
    Let $A\subseteq \mathcal{X}$, $B\subseteq \mathcal{Y}$. We say that $B$ is \textbf{Borel reducible} to $A$ if there is a Borel map $f: \mathcal{Y}\rightarrow \mathcal{X}$ satisfying $f^{-1}(A)=B$.
\end{definition}
\begin{fact}
    \label{analityczny_z_analitycznego}
    If an analytic set $B$ is Borel reducible to $A$ and $B$ is $\analytic$-complete, then $A$ is $\analytic$-complete.
\end{fact}

Note that existence of analytic non-Borel sets and closure of Borel sets under Borel maps implies that all $\analytic$-complete sets are not Borel. Moreover, in order to apply \Cref{analityczny_z_analitycznego}, we need an example of $\analytic$-complete set. Such an example can be found among trees over $\omega$.
\begin{definition}
    A set $T\subseteq\seqN$ is a \textbf{tree over $\omega$} if
    $$\left(\forall \sigma\in\seqN\right)\left(\forall \tau\in\seqN\right)\left(\sigma\in T\land \tau\subseteq\sigma\implies\tau\in T\right).$$
    Set of all trees over $\omega$ will be denoted by $\trees$. A \textbf{body} of a tree $T$ is a set
    $$[T]=\left\{\sigma\in\baire: \left(\forall n\in\omega\right)\left(\sigma\upharpoonright n\in T\right)\right\}.$$
\end{definition}

Using above definition set $\trees$ can be seen as a $G_\delta$ subset of $P(\seqN)$. Therefore, $\trees$ is a Polish space. By $\illfounded$ let us denote the collection of all ill-founded trees over $\omega$, i.e. all trees with non-empty body. It occurs that $\illfounded$ is an example of $\analytic$-complete set we were looking for (see e.g. \cite[Example 4.2.1]{srivastava}).

Given an analytic set, proving its analytic completeness is a fundamental way of showing that it is not Borel. As shown in \cite{languages}, $\analytic$-complete (or, in this case, rather $\coanalytic$-complete) sets can be used in not necessarily set theoretic context. In the paper, authors investigate properties of regular languages of thin trees. In particular, they are interested in descriptive properties of such languages. One of their result is that regular language, which does not fulfil some definability condition (so called not WMSO-definable language), is $\coanalytic$-complete.

Naturally, $\analytic$-complete sets can be useful in more set theory-related research. Like in \cite{izomorfizmy_c0}, where class of all Banach spaces isomorphic to $c_0$ is considered. The main result of the work states that this class is a complete analytic set (with respect to Effros Borel structure), so it can not be Borel.

In \cite{colorings}, a class of coloring problems induced by actions of countable group on Polish spaces is studied. It is shown, that the set of such coloring problems, which additionally have Baire measurable solution for a particular free action $\alpha$, is $\analytic$-complete (when $\alpha$ is not trivial).

In this paper we would like to examine descriptive complexity of sequences with increasing subsequence, seen as a subset of $\baire$ (or other space homeomorphic to it). The motivation comes from classical example of Lusin (which can be found in \cite[27.2]{kechris}):

\begin{theorem}[Lusin]
	\label{lusin}
	Let $\mid$ be a division of positive natural numbers $\mathbb{N}$. Set
	$$L=\{y\in \mathbb{N}^\omega: (\exists k_0<k_1<\ldots)(\forall i\in\omega)(y(k_i)\mid y(k_{i+1}))\}.$$
	$L$ is a $\analytic$-complete subset of $\mathbb{N}^\omega$. 
\end{theorem}

We want to study the descriptive complexity of sets defined in a similar fashion. Assume that $X$ is countable set and $R$ is a relation on $X$. Define
$$L_{(X,R)}=\{y\in X^\omega: (\exists k_0<k_1<\ldots)(\forall i\in\omega)(y(k_i) R y(k_{i+1}))\}.$$



In the next section we will provide some basic facts and discuss the complexity of $L_{(X,R)}$ for various examples of $(X,R)$. 
We will focus mainly on the case of posets, i.e sets equipped with a relation which is reflexive, symmetric and transitive. 
Later we will consider linear orders and give a characterization of those for which the set $L_{(X,R)}$ is $\analytic$-complete.
	\section{Basic examples}\label{basics}

First we shall observe that projective class of $L_{(X,R)}$ can not exceed $\analytic$.
\begin{fact}
	Assume that $R\subseteq X\times X$ and $|R|\le \aleph_0$. Then the set $L_{(X,R)}$ is analytic.
\end{fact}
\begin{proof}
	Let us define
	$$B_{(X,R)}=\{(k,y)\in\baire\times X^\omega: (\forall i\in\omega)(k_i<k_{i+1}\land y(k_i) R y(k_{i+1}))\}.$$
	 Notice that  $B_{(X,R)}$ is Borel. Indeed,
	\begin{align*}
		B_{(X,R)}&=\bigcap_{i\in\omega}\{(k,y)\in\baire\times X^\omega: k_i<k_{i+1}\land y(k_i) R y(k_{i+1})\}\\
		&=\bigcap_{i\in\omega}\bigcup_{a\in\omega}\bigcup_{b>a}\big(\{k\in\baire: k_i=a, k_{i+1}=b\}\times\{y\in X^\omega: y(k_i)R y(k_{i+1})\}\big)\\
		&=\bigcap_{i\in\omega}\bigcup_{a\in\omega}\bigcup_{b>a}\left[\{k\in\baire: k_i=a, k_{i+1}=b\}\times\left(\bigcup_{(y_1, y_2)\in R}\{y\in X^\omega: y(a)=y_1, y(b)=y_2\}\right)\right]
	\end{align*}
	and $R$ is countable. 
	Clearly, $L_{(X,R)}=\pi_{X^\omega}[B_{(X,R)}]$ is a projection of a Borel set. So $L_{(X,R)}$ is analytic.
\end{proof}
In case when $X$ is finite, every sequence of elements of $X$ contains a constant subsequence. This observation gives us following:
\begin{fact}
	If $X$ is finite and $R$ is a reflexive relation on $X$, then $L_{(X,R)}=X^\omega$.
\end{fact}
\begin{fact}
	\label{lusin_rownosc}
	For a countable set  $X$ define  $\Delta_X=\{(x,x): x\in X\}$. Then $L_{(X,\Delta_X)}$ is Borel.
\end{fact}
\begin{proof}
	\begin{align*}
		L_{(X,\Delta_X)}&=\{y\in X^\omega: (\exists k_0<k_1<\ldots)(\forall i\in\omega)(y(k_i)= y(k_{i+1}))\}\\
		&=\{y\in X^\omega: (\exists x\in X)(\exists k_0<k_1<\ldots)(\forall i\in\omega)(y(k_i)=x)\}\\
		&=\{y\in X^\omega: (\exists x\in X)(\forall n\in\omega)(\exists k>n)(y(k)=x)\},
	\end{align*}
	what clearly gives us that $L_{(X,\Delta_X)}$ is $G_{\delta\sigma}$.
\end{proof}

\begin{question}
    What is the precise complexity of $L_{(X,\Delta_X)}$? Is it not $F_{\sigma\delta}$?
\end{question}

Notice that for any poset $(X, \leq_X)$
above result shows that, in order to identify projection class of $L_{(X,\leq_X)}$, we can focus on analyzing strictly increasing sequences. 
$$L_{(X,\leq_X)}=L_{(X,\Delta_X)}\cup\{y\in X^\omega: (\exists k_0<k_1<\ldots)(\forall i\in\omega)(y(k_i)<_X y(k_{i+1}))\}.$$

Now we can move to classification of linear orders in this problem. Because in well orderings there are no infinite decreasing subsequences, below fact follows:
\begin{fact}
	Assume that $\leq_X$ is a well ordering on (countable) $X$. Then $L_{(X,\le_X)}=X^\omega$.
\end{fact}

Now let us consider the set of integers equipped with a standard order $\le$. It is probably one of the simplest linear orders which is not a well ordering.
\begin{fact}
	 The set $L_{(\mathbb{Z},<)}$ is $G_{\delta}$ and not $F_\sigma$.
\end{fact}
\begin{proof}
	Observe that every strictly increasing sequence of integers is unbounded. So we can write
	\begin{align*}
		&L_{(\mathbb{Z}, <)}=\{y\in \mathbb{Z}^\omega: (\exists k_0<k_1<\ldots)(\forall i\in\omega)(y(k_i)< y(k_{i+1}))\}\\
		&=\{y\in\mathbb{Z}^\omega: (\forall n\in\mathbb{Z})(\exists k\in\omega)(y(k)>n)\}\\
		&=\bigcap_{n\in\mathbb{Z}}\bigcup_{k\in\omega}\{y\in\mathbb{Z}^\omega: y(k)>n\}\\
		&=\bigcap_{n\in\mathbb{Z}}\bigcup_{k\in\omega}\bigcup_{m>n}\{y\in\mathbb{Z}^\omega: y(k)=m\},
	\end{align*}
	which is clearly a $G_\delta$ set as $\{y\in\mathbb{Z}^\omega: y(k)=m\}$ is clopen.

    Now note that both $L_{(\mathbb{Z}, <)}$ and $L_{(\mathbb{Z}, <)}^c$ have empty interiors (since they cannot include any base open set). Therefore $L_{(\mathbb{Z}, <)}^c$ is meager (as an $F_\sigma$ set without interior). If $L_{(\mathbb{Z}, <)}$ is an $F_\sigma$ set, it is also meager contradicting Baire category theorem. 
\end{proof}
From the observation made after \Cref{lusin_rownosc} and above fact we obtain following corollary:
\begin{corollary}
	 The set $L_{(\mathbb{Z},\leq)}$ is Borel.
\end{corollary}

	\section{Main results}

One of the tools in recognizing $\analytic$-complete sets among the sets of the form $L_{(X,R)}$ is the following observation.
\begin{theorem}
	\label{generowanie}
	Suppose $(X, \leq_X)$, $(Y, \leq_Y)$ are posets and $\varphi: X\rightarrow Y$ satisfies the following condition for every $(x_n)_{n\in\omega}\in X^\omega$:
	$$(x_n)_{n\in\omega} \text{ contains $\leq_X$-increasing subsequence} \Leftrightarrow (\varphi(x_n))_{n\in\omega} \text{ contains $\leq_Y$-increasing subsequence}.$$
 If $L_{(X,\leq_X)}$ is $\analytic$-complete, then $L_{(Y,\leq_Y)}$ is $\analytic$-complete too.
\end{theorem}
\begin{proof}
	Let $Z$ be a Polish space and $A\subseteq Z$ be an analytic set. There is a Borel map $f: Z\rightarrow X^\omega$ such that $f^{-1}[L_X]=A$. We need a Borel map $h:Z\rightarrow Y^\omega$ satisfying $h^{-1}[L_Y]=A$.
	
	Define $g: X^\omega\rightarrow Y^\omega$ with formula
	$$g(x)(n)=\varphi(x_n).$$
	Clearly, $g$ is continuous, so $h=g\circ f$ is Borel. For the thesis it is sufficient to show that $g^{-1}[L_Y]=L_X$.
	\begin{align*}
		x\in L_{(X,\leq_X)}&\Leftrightarrow x\text{ contains a $\leq_X$-increasing subsequence}\\
		&\Leftrightarrow g(x)\text{ contains a $\leq_Y$-increasing subsequence}\\
		&\Leftrightarrow g(x)\in L_{(Y,\leq_Y)}\Leftrightarrow x\in g^{-1}[L_{(Y,\leq_Y)}]
	\end{align*}
\end{proof}
\begin{corollary}
	\label{podzbior_analityczny}
	Assume that  $X\subseteq Y$, $S\subseteq Y\times Y$, $R=S\cap (X\times X)$ and $L_{(X, R)}$ is $\analytic$-complete. Then  $L_{(Y,S)}$ is $\analytic$-complete, too. 
\end{corollary}
\begin{proof}
    It is enough to take $\varphi(x)=x$ in \Cref{generowanie}.
\end{proof}

\begin{corollary}
	\label{izomorfizm_analityczny}
	Assume that $(X, \leq_X)$ and $(Y, \leq_Y)$ are isomorphic posets and $L_{(X, \leq_X)}$ is $\analytic$-complete. Then $L_{(Y, \leq_Y)}$ is $\analytic$-complete, too. 
\end{corollary}
\begin{proof}
    To see this, put order isomorphism as $\varphi$ in \Cref{generowanie}.
\end{proof}

Let us now show an example of $\analytic$-complete set based on a space of finite sequences of naturals. 
\begin{theorem}
	\label{ciagi_nad_omega}
 The set $L_{(\seqN, \subseteq)}$ is $\analytic$-complete.
\end{theorem}
\begin{proof}
   To prove that $L_{(\seqN, \subseteq)}$ is $\analytic$-complete we will construct a continuous function $f:\trees\rightarrow(\seqN)^\omega$ such that $f^{-1}[L_{(\seqN, \subseteq)}]=\illfounded$ 
   First, fix an enumeration $\{\sigma_n: n\in\omega\}$ of $\seqN$ satisfying the following condition
	$$\sigma_n\subseteq\sigma_m\Rightarrow n\leq m.$$
	Now we can define the function $f$:
	$$f(T)(n)=\left\{\begin{array}{ll}
		\sigma_n, & \sigma_n\in T \\
		1^n0, & \sigma_n\notin T
	\end{array}\right. .$$
	Clearly, if $T\in\illfounded$, then $f(T)$ contains $\subseteq$-increasing subsequence, hence $f(T)\in L_{(\seqN, \subseteq)}$. To prove the opposite implication, let $a\in L_{(\seqN, \subseteq)}$, $a_{i_0}\subseteq a_{i_1}\subseteq a_{i_2}\subseteq\ldots$, $i_0<i_1<i_2<\ldots$. Take any $T\in f^{-1}(a)$. Notice that at most one of $a_{i_0},a_{i_1},\ldots$ can be of the form $1^n0$ for some $n\in\omega$, so without loss of generality all of them are elements of $T$ and form a strictly increasing sequence. But such a sequence of elements of $T$ builds a branch in $T$, so $T\in\illfounded$.
\end{proof}
\begin{theorem}
 The set  $L_{(\seq2, \subseteq)}$ is $\analytic$-complete.
\end{theorem}
\begin{proof}
	We will use \Cref{ciagi_nad_omega,generowanie}. First, define function $f:\omega\rightarrow\seq2$ with formula
	$$f(n)=a_0a_0a_1a_1\ldots a_ma_m,$$
	where $n=(a_0a_1\ldots a_m)_2$, i.e. $a_0a_1\ldots a_m$ is a binary reprezentation of $n$. Now consider a function $\varphi: \seqN\rightarrow\seq2$ defined as:
	$$\varphi(b_0b_1\ldots b_n)= f(b_0)\string^01\string^f(b_1)\string^01\string^\ldots\string^01\string^f(b_n)\string^01.$$
	$\varphi$ and $\varphi^{-1}$ are both increasing (with respect to ordering defined by $\subseteq$), so $\varphi$ fulfills requirements of \Cref{generowanie}. Hence, the thesis holds.
\end{proof}

On $\seqN$ let us define an ordering $\RLleq$ with the formula
$$x\RLleq y\iff (\exists n\in\omega)(x=y\upharpoonright n \lor (x\upharpoonright n=y\upharpoonright n \land x(n)>y(n))).$$
Relation $\RLleq$ can be seen as the lexicographical order on $\seqN$ with modification, that order on $\omega$ is reversed.

\begin{theorem}
	\label{RLorder_theorem}
	 $L_{(\seqN,\RLleq)}$ is $\analytic$-complete.
\end{theorem}
\begin{proof}
	We will construct a continuous function $f:\trees\rightarrow(\seqN)^\omega$ such that $f^{-1}[L_X]=\illfounded$. First, fix an enumeration $\{\sigma_n: n\in\omega\}$ of $\seqN$ like in proof of \Cref{ciagi_nad_omega}. Now we can define function $f$:
	$$f(T)(n)=\left\{\begin{array}{ll}
		\sigma_n, & \sigma_n\in T \\
		1^n0, & \sigma_n\notin T
	\end{array}\right. .$$
	Clearly, if $T\in \textrm{IF}_\omega$, then $f(T)$ contains $\RLleq$-increasing subsequence, hence $f(T)\in L_X$. Let $a\in L_X$, $a_{i_0}\RLless a_{i_1}\RLless a_{i_2}\RLless\ldots$, $i_0<i_1<i_2<\ldots$. Take any $T\in f^{-1}(a)$. Notice that at most one of $a_{i_0},a_{i_1},\ldots$ can be of the form $1^n0$ for some $n\in\omega$ (as $0\RLmore10\RLmore110\RLmore\ldots$), so without loss of generality all of them are elements of $T$ and $|a_{i_0}|>0$.
	
	Since $a_{i_0}$ is $\RLleq$-smallest of $a_{i_0},a_{i_1},\ldots$, it must be the case that $a_{i_0}(0)\geq a_{i_j}(0)$ for all $j\in\omega$. Therefore there are only finitely many possible values for $a_{i_j}(0)$, so infinitely many of them start with the same number, say $\tau(0)$. Analogically, from all $a_{i_j}$ which start with $\tau(0)$ infinitely many have the same number at position $1$, say $\tau(1)$. Continuing this way we obtain $\tau\in\baire$ such that for every $n\in\omega$ there is $j\in\omega$ satisfying
	$$\tau\upharpoonright n\preceq a_{i_j},$$
	so (because $T$ is a tree and $a_{i_j}\in T$) $\tau\upharpoonright n\in T$. It follows that $\tau$ is an infinite branch of $T$.
\end{proof}

Now let us focus on rational numbers with standard ordering. Notice that this poset can be seen as ''the most complicated'' among countable linear orderings, since it contains an isomorphic copy of any countable linear order. Firstly, we shall see that $(\mathbb{Q}, \leq)$ generates $\analytic$-complete set, opposed to linear orderings investigated in \Cref{basics}.
\begin{theorem}\label{wymierne}
	The set  $L_{(\mathbb{Q}, \leq)}$ is $\analytic$-complete.
\end{theorem}
\begin{proof}
    Define a function $\varphi:\seqN\rightarrow\mathbb{Q}$ with formula ($\varphi(\varepsilon)=0$)
    $$\varphi(a_0a_1a_2\ldots a_n)=(0.\underbrace{00\ldots0}_{a_0}1\underbrace{00\ldots0}_{a_1}1\underbrace{00\ldots0}_{a_2}1\ldots\underbrace{00\ldots0}_{a_n}1)_2.$$
    Considering $\RLleq$ on $\seqN$, $\varphi$ and $\varphi^{-1}$ are clearly increasing. Therefore $\varphi$ is an order isomorphism between $(\seqN, \RLleq)$ and $(\varphi(\seqN), \leq)$. Thus, the thesis follows from \Cref{podzbior_analityczny,izomorfizm_analityczny}.
\end{proof}

Next, we would like to characterize all linear ordering which yields an $\analytic$-complete set. The following theorem, as explained later, will serve as a main tool in our task.

\begin{theorem}\label{charakteryzacja}
	Suppose $X\subseteq\mathbb{Q}\cap[0,1]$, $\leq_X=\leq\cap (X\times X)$. Let $\overline{X}$ be the closure of $X$ in the Euclidean topology. We have two possible cases.
	\begin{enumerate}
		\item If $|\overline{X}|=\omega$, then $L_{(X,\leq_X)}$ is Borel.
		\item If $|\overline{X}|=\mathfrak{c}$, then $X$ contains $\leq$-dense subset.
	\end{enumerate}
	
\end{theorem}
\begin{proof}
	Firstly, consider the case $|\overline{X}|=\omega$. For $g\in[0,1]$ define
	$$L_g=\{y\in X^\omega: (\forall a\in X)((a<g)\rightarrow (\forall N\in\omega)(\exists n>N)(a<y_n\leq g))\}.$$
	We want to show that
	$$L_{(X,\leq_X)}=\bigcup\{L_g:\, g\in\overline{X}\}.$$
	
	Take any $g\in[0,1]$, for which $L_g\ne\emptyset$, and $y\in L_g$. Let $N_0=\max\{n: y_n=g\}$ (if such $n$ does not exist, $y$ contains a constant subsequence, hence $y\in L_{(X,\leq_X)}$) and take $k_0=N_0+1$. From the definition of $L_g$ there is $k_1>k_0$ such that $y_{k_0}<y_{k_1}<g$ (because $y_{k_0}\in X$, $y_{k_0}<g$). Analogically we can find $k_2>k_1$ satisfying $y_{k_1}<y_{k_2}<g$. Continuing this way we obtain a sequence $k_0<k_1<k_2,\ldots$ defining an increasing subsequence $(y_{k_i})_{i\in\omega}$ of $y$.
	
	On the other hand, when $y\in L_{(X,\leq_X)}$, it contains a non-decreasing subsequence. But this subsequence is bounded (like the whole $X\subseteq[0,1]$), so it converges to some $g\in\overline{X}$. Thus, $y\in L_g$.
	
	Note that, since $X$ is countable, $L_g$ is a Borel set for every $g\in[0,1]$. So $L_{(X,\leq_X)}$, as a countable union of $L_g$'s, is also Borel.\\
	
	Now let us focus on the second case, i.e. $|\overline{X}|=\mathfrak{c}$. First, observe that if $[a,b]\in\overline{X}$ for some $a,b\in[0,1]$, $a<b$, then $X\cap[a,b]$ is $\leq$-dense. Thus, assume that $\overline{X}$ does not contain an interval. There is a perfect nowhere dense set $C\subseteq\overline{X}$. Without loss of generality we can presume that $0,1\in C$ (otherwise we consider interval $[a,b]$, where $a=\min C$, $b=\max C$). We will represent $C$ i more convenient way. To do this we will inductively construct a family $\{C_\sigma:\, \sigma\in \seq2\}$ of closed intervals and a family $\{U_\sigma:\, \sigma\in\seq2\}$ of open intervals (similarly to the  classical construction of the Cantor set).
 
  We start with $C_\varepsilon=[0,1]$. Since $C$ is nowhere dense, we can take a maximal open interval $U_\varepsilon=(a_\varepsilon,b_\varepsilon)$, $U_\varepsilon\subseteq[0,1]$, disjoint with $C$. Hence, $C\subseteq[0,a_\varepsilon]\cup[b_\varepsilon,1]$. Next we see that $a_\varepsilon\ne0$ (and $b_\varepsilon\ne1$), because otherwise $0\in C$ would be an isolated point of perfect set $C$. Moreover, from maximality of $U_\varepsilon$, $a_\varepsilon,b_\varepsilon\in C$. Let us denote $[0,a_\varepsilon]=C_{(0)}$, $[b_\varepsilon,1]=C_{(1)}$, $l_{(0)}=0$, $p_{(0)}=a_\varepsilon$, $l_{(1)}=b_\varepsilon$, $p_{(1)}=1$.
  
  Assume now that  $C_\sigma=[l_\sigma,p_\sigma]$ has been already constructed for some  $\sigma\in\seq2$. Analogically as in the previous point we choose a maximal open interval $U_\sigma=(a_\sigma,b_\sigma)\subseteq[l_\sigma,p_\sigma]$ disjoint with $C$. We denote $l_{\sigma\string^0}=l_\sigma$, $p_{\sigma\string^0}=a_\sigma$, $l_{\sigma\string^1}=b_\sigma$, $p_{\sigma\string^1}=p_\sigma$.

	Taking $$C_n=\bigcup\{C_\sigma:\, \sigma\in 2^{<\omega} \\ |\sigma|=n\}$$
	it is clear that $C=\bigcap_{n\in\omega}C_n$.
 
	Therefore, if we put $\mathcal{U}=\{U_\sigma: \sigma\in2^{<\omega}\}$,
	\begin{equation}\label{rownanie_na_C}
	    C=[0,1]\backslash\bigcup\mathcal{U}.
	\end{equation}
	Furthermore
	\begin{equation}\label{konce}
		\{l_\sigma: \sigma\in\seq2\}\cup\{p_\sigma: \sigma\in\seq2\}\subseteq C.
	\end{equation}
We will now consider two possibilities. First, $X$ contains dense-in-itself set and second, $X$ does not contain dense-in-itself set.  
    
    In the first situation, $X$ contains a dense-in-itself set. Without loss of generality $X$ is dense-in-itself (otherwise we repeat above construction for closure of this dense-in-itself subset of $X$). $\overline{X}$ is then a perfect set and does not contain an interval, thus is nowhere dense. Hence, we can put $C=\overline{X}$. Consider a set
    $$P=X\backslash\{p_\sigma: l_{\sigma+1}\in X\},$$
    where $\sigma+1$ is a successor of $\sigma\in 2^n$ in lexicographical order on $2^n$ (in other words binary adding $1$ to $\sigma$ and $111\ldots11+1$ does not exist). We claim that $P$ is $\leq$-dense. Take any $a,b\in P$, $a<b$. From the above construction there is $\sigma\in\seq2$ such that
    $$a\leq p_\sigma<l_{\sigma+1}\leq b.$$
        
    First, assume that $a<p_\sigma$ and $l_{\sigma+1}<b$. If $p_\sigma\in X$ or $l_{\sigma+1}\in X$, claim clearly holds. Otherwise $p_\sigma\in\overline{X}$, so there is $x\in X$ close to $p_\sigma$. Therefore
        $$a<x<b.$$
        
    Second, presume that $a=p_\sigma$. From definition of $P$, $l_{\sigma+1}<b$. Since $X$ is dense-in-itself, there is $x\in X$ satisfying $l_{\sigma+1}<x<b$. If $x\in P$, claim holds. If not, $x=p_\tau<l_{\tau+1}\leq b$ for some $\tau\in\seq2$. Again, there is $y\in X$ such that $a<y<x$. When $y\in P$, claim holds. Otherwise $y=p_\psi$ and $l_{\psi+1}\in P$ for some $\psi\in\seq2$. But then
        $$a<p_\psi<\underbrace{l_{\psi+1}}_{\in P}<x<b.$$
        
    The case $a<p_\sigma$, $l_{\sigma+1}=b$ is analogous to previous one.
    
    Finally, consider a situation when $X$ does not contain any dense-in-itself set. We start by proving that
    \begin{equation}\label{zycie_w_dziurach}
        \overline{X\backslash C}\supseteq C.
    \end{equation}
    Suppose not, so there is $z\in C$ such that $z\notin\overline{X\backslash C}$. There exists an open interval $U=(l,p)\ni z$ disjoint with $X\backslash C$. As $z\in C\subseteq\overline{X}$, $X\cap U\ne\emptyset$ and $X\cap U\subseteq C\cap U$. We claim that $X\cap U$ is dense-in-itself. Take any $x\in X\cap U$ and $\varepsilon>0$. We want to find $y\in(x-\varepsilon, x+\varepsilon)\cap U\cap X$. Since $x\in C\cap U$, it exists $c\in(x-\varepsilon, x+\varepsilon)\cap U\cap C$. Because $c\in C$, we can find $y\in X$ close to $c$, especially $y\in(x-\varepsilon, x+\varepsilon)\cap U\cap X$. Therefore $X\cap U$ is dense-in-itself, which contradicts assumption that $X$ does not contain such a set.
    
    From \cref{rownanie_na_C} we see that
    $$X\backslash C=X\cap\bigcup\mathcal{U}=\bigcup\{X\cap U_\sigma: U_\sigma\in\mathcal{U}\}.$$
    Let $Y$ be a selector of family $\{X\cap U_\sigma: U_\sigma\in\mathcal{U}\}\backslash\{\emptyset\}$. We claim that $Y$ is $\leq$-dense. Take any $a,b\in Y$, $a<b$. Take $\sigma,\psi\in\seq2$, $\sigma\ne\psi$, such that $a\in U_\sigma$, $b\in U_\psi$. There is $\tau\in\seq2$ satisfying $C_\tau=[l_\tau,p_\tau]\subseteq[l_{\sigma\string^1}, p_{\psi\string^0}]$ and $C_\tau\ne[l_{\sigma\string^1}, p_{\psi\string^0}]$. Suppose that $l_\tau\ne l_{\sigma\string^1}$ (case when $p_\tau\ne p_{\psi\string^0}$ is analogous). From \ref{konce} it follows that $l_\tau\in C$, so (from \ref{zycie_w_dziurach}) $l_\tau\in\overline{X\backslash C}$. Thus, there is a sequence from $X\cap\bigcup\mathcal{U}$ convergent to $l_\tau$. Hence, there is $\phi\in\seq2$ satisfying
    $$U_\phi\subseteq[l_{\sigma\string^1}, p_{\psi\string^0}], X\cap U_\phi\ne\emptyset.$$
    Therefore there exists $x\in Y\cap X\cap U_\phi$. Clearly,
    $a<x<b.$
\end{proof}

\begin{theorem}\label{linki}
    Let $(X,\leq_X)$ be a linear order. $L_{(X,\leq_X)}$ is $\analytic$-complete if and only if $X$ contains $\leq_X$-dense subset.
\end{theorem}
\begin{proof}
    First, note that every linear order can be embedded into $(\mathbb{Q}\cap[0,1], \leq)$. Therefore, we can assume without loss of generality that $X\subseteq\mathbb{Q}\cap[0,1]$, $\leq_X=\leq$.
    
    Suppose that $L_{(X,\leq_X)}$ is $\analytic$-complete. From \Cref{charakteryzacja}, $X$ contains $\leq$-dense subset. On the other hand, if $X$ contains $\leq$-dense subset $Y$, $Y$ is order-isomorphic to $\mathbb{Q}$ (since $\mathbb{Q}$ is the only, up to isomorphism, countable dense linear order). From \Cref{podzbior_analityczny,izomorfizm_analityczny}, $L_{(X,\leq_X)}$ is $\analytic$-complete.
\end{proof}
\begin{corollary}
    Let $\lexq$ be the lexicographical order on $\seq2$. Then $L_{(\seq2,\lexq)}$ is $\analytic$-complete.
\end{corollary}
\begin{proof}
    Consider set $A=\{x\string^1: x\in\seq2\}$, i.e. the set of all sequences ending with $1$. We claim that this set is $\lexq$-dense. Take any $\sigma\string^1\lex\tau\string^1$, $\sigma,\tau\in\seq2$.
        
    If $\sigma\string^1\subseteq\tau\string^1$, then
        $$\sigma\string^1\lex\sigma\string^10^{|\tau|-|\sigma|}1\lex\tau\string^1.$$
        
    Otherwise
        $$\sigma\string^1\lex\sigma\string^11\lex\tau\string^1.$$
    Hence, by \Cref{linki}, $L_{(\seq2,\lexq)}$ is $\analytic$-complete.
\end{proof}

\begin{question}
    What is the characterisation of countable posets $(X,\le_X)$ such that $L_{(X,\le_X)}$ is $\analytic$-complete?
\end{question}

\end{document}